\documentclass[12pt,fleqn,leqno]{amsart}
\usepackage{amsfonts,amssymb}
\usepackage{mathrsfs}
\usepackage[bb=boondox]{mathalfa}
\usepackage{enumitem}

\usepackage[svgnames]{xcolor}
\usepackage[colorlinks,linkcolor=blue,citecolor=Green]{hyperref}
\usepackage{titlesec}

\titleformat{\section}[block]
 {\bfseries}
 {\thesection.}
 {\fontdimen2\font}
 {}

\newcommand{\periodafter}[1]{#1.}

\titleformat{\subsection}[runin]
 {\bfseries}
 {\thesubsection.}
 {\fontdimen2\font}
 {\periodafter}

\setlist{noitemsep}
\setenumerate{labelindent=\parindent,label=\upshape{(\alph*)}}

\newtheorem{theorem}{Theorem}[section]
\newtheorem{corollary}[theorem]{Corollary}
\newtheorem{proposition}[theorem]{Proposition}
\newtheorem{lemma}[theorem]{Lemma}

\DeclareMathOperator{\N}{\mathbb{N}}
\DeclareMathOperator{\R}{\mathbb{R}}
\DeclareMathOperator{\uhr}{\upharpoonright}
\DeclareMathOperator{\J}{\mathbb{J}}
\DeclareMathOperator{\ord}{Ord}

\renewcommand{\emptyset}{\varnothing}
\numberwithin{equation}{section}

\begin{document}

\author{Valentin Gutev}

\address{Department of Mathematics, Faculty of Science, University of
     Malta, Msida MSD 2080, Malta}

\email{\href{mailto:valentin.gutev@um.edu.mt}{valentin.gutev@um.edu.mt}}



\title{Insertion of Continuous Set-Valued Mappings}
\begin{abstract}
  An interesting result about the existence of ``intermediate''
  set-valued mappings between pairs of such mappings was obtained by
  Nepomnyashchii. His construction was for a paracompact domain, and
  he remarked that his result is similar to Dowker's insertion theorem
  and may represent a generalisation of this theorem. In the present
  paper, we characterise the $\tau$-paracompact normal spaces by this
  set-valued ``insertion'' property and for $\tau=\omega$, i.e.\ for
  countably paracompact normal spaces, we show that it is indeed
  equivalent to the mentioned Dowker's theorem. Moreover, we obtain a
  similar result for $\tau$-collectionwise normal spaces and show that
  for normal spaces, i.e.\ for $\omega$-collectionwise normal spaces,
  our result is equivalent to the Kat\v{e}tov-Tong insertion
  theorem. Several related results are obtained as well.
\end{abstract}

\dedicatory{Dedicated to the memory of Mitrofan Choban}

\date{\today}
\maketitle

\section{Introduction}

All spaces in this paper are infinite Hausdorff topological
spaces. For a set $Y$, we use $2^Y$ to denote the family of all
nonempty subsets of $Y$. For a space $Y$, let
\[
  \mathscr{F}(Y)=\left\{S\in 2^Y:S\ \text{is closed}\right\}.
\]
The following subfamilies of $\mathscr{F}(Y)$ will play an important
role in this paper:
\[
  \mathscr{C}(Y)=\{S\in \mathscr{F}(Y): S\ \text{is
    compact}\}\quad\text{and}\quad 
  \mathscr{C}'(Y)=\mathscr{C}(Y)\cup\{Y\}.
\]
Moreover, we will use the subscript ``$\mathbf{c}$'' to denote the
connected members of any one of the above families; namely,
$\mathscr{F}_\mathbf{c}(Y)$ for the connected members of
$\mathscr{F}(Y)$; $\mathscr{C}_\mathbf{c}(Y)$ for those of
$\mathscr{C}(Y)$; and
$\mathscr{C}_\mathbf{c}'(Y)=\mathscr{C}_\mathbf{c}(Y)\cup \{Y\}$
provided that $Y$ is itself connected.  Let us remark that in the
setting of a topological vector space $Y$, some authors have used the
subscript ``\textbf{c}'' in a more restricted sense, namely for the
convex members of any of the above families. In this paper, convex
sets will not be used explicitly, hence it will not cause any
misunderstandings. \medskip

For spaces $X$ and $Y$, a set-valued mapping $\Phi:X\to 2^Y$ is
\emph{lower semi-continuous}, or \emph{l.s.c.}, if the set
$\Phi^{-1}(U)=\{x\in X: \Phi(x)\cap U\neq \emptyset\}$ is open in $X$
for every open $U\subseteq Y$; and $\Phi$ is \emph{upper
  semi-continuous}, or \emph{u.s.c.}, if $\Phi^{-1}(F)$ is closed in
$X$ for every closed $F\subseteq Y$. Equivalently, $\Phi$ is u.s.c.\
if
\[
\Phi^\#(U)=X\setminus \Phi^{-1}(Y\setminus U)= \big\{x\in X:
\Phi(x)\subseteq U\big\}
\]
is open in $X$, for every open $U\subseteq Y$.  A set-valued mapping
is \emph{continuous} if it is both l.s.c.\ and u.s.c. A mapping
$\theta:X\to 2^Y$ is a \emph{selection} (or, a \emph{set-valued
  selection}) for $\Phi:X\to 2^Y$ if $\theta(x)\subseteq\Phi(x)$ for
all $x\in X$. To designate that $\theta$ is a selection for $\Phi$, we
will often simply write $\theta\leqslant\Phi$. Moreover, to bring the
analogy closer to usual functions, we will also write
$\theta<\Phi$ provided that $\theta\leqslant \Phi$ and
$\Phi(x)\setminus \theta(x)\neq \emptyset$, for all $x\in X$.  In
\cite{nepomnyashchii:86}, Nepomnyashchii showed that if $X$ is a
paracompact space, $(Y,\rho)$ is a complete metric space,
$\Phi:X\to \mathscr{F}_\mathbf{c}(Y)$ is l.s.c.\ such that
$\{\Phi(x):x\in X\}$ is uniformly equi-$LC^0$ and
$\theta:X\to \mathscr{C}(Y)$ is u.s.c.\ with $\theta\leqslant\Phi$,
then there exists a continuous mapping
$\varphi:X\to \mathscr{C}_\mathbf{c}(Y)$ such
that~$\theta\leqslant \varphi\leqslant\Phi$. For the definition of
``uniformly equi-$LC^0$'', see the next section.\medskip

In this paper, we will show that such intermediate set-valued mappings
can characterise paracompactness and collectionwise normality. To
state our results, we briefly recall some terminology.  For an
infinite cardinal number $\tau$, a space $X$ is
\emph{$\tau$-paracompact} \cite{zbMATH03161751,morita:60} if each open
cover of $X$ of cardinality $\leq\tau$, has a locally finite open
refinement. If $\omega$ is the first infinite ordinal, then
$\omega$-paracompactness is nothing else but countable
paracompactness. By definition, a space $X$ is paracompact if it is
$\tau$-paracompact for any cardinal number $\tau$. Paracompactness of
$X$ implies normality of $X$. However, there are many examples of
countably paracompact spaces which are not normal. Also, there are
simple examples of $\tau$-paracompact spaces which are not
$\tau^+$-paracompact \cite{morita:60}, where the cardinal $\tau^+$ is
the immediate successor of $\tau$.  Finally, for a set $\mathscr{A}$,
we will use $\J(\mathscr{A})$ to denote the \emph{metrizable hedgehog
  of spininess} $|\mathscr{A}|$ obtained from $\mathscr{A}$ (see
\cite[Example 4.1.5]{engelking:89}). Recall that the point set of
$\J(\mathscr{A})$ is $\{0\}\cup (0,1]\times \mathscr{A}$, while the
metric $d$ on $\J(\mathscr{A})$ is defined by
\begin{equation}
  \label{eq:choban-sel-v7:1}
  d(0,\langle s,\alpha\rangle)=s\quad \text{and}\quad 
  d (\langle s,\alpha\rangle, \langle t,\beta\rangle) =
  \begin{cases}
    |s-t|,  &\text{if }\alpha=\beta, \\
    s+t, &\text{if }\alpha \ne \beta.
  \end{cases}
\end{equation}
Throughout this paper, $\J(\mathscr{A})$ will be always endowed with
this metric. The following theorem will be proved.

\begin{theorem}
  \label{theorem-choban-sel-v6:1}
  For an infinite cardinal number $\tau$ and a space $X$, the
  following conditions are equivalent\textup{:}
  \begin{enumerate}
  \item\label{item:choban-sel-v6:1} $X$ is normal and
    $\tau$-paracompact. 
  \item \label{item:choban-sel-v6:2} If
    $\Phi:X\to \mathscr{F}_\mathbf{c}(\J(\tau))$ is l.s.c.,
    $\theta:X\to \mathscr{C}(\J(\tau))$ is u.s.c.\ and $\theta<\Phi$,
    then there exists a continuous mapping
    ${\varphi:X\to \mathscr{C}_\mathbf{c}(\J(\tau))}$ such that
    $\theta<\varphi\leqslant \Phi$.
  \end{enumerate}
\end{theorem}

Theorem \ref{theorem-choban-sel-v6:1} is not only similar, but also
represents a natural generalisation of Dowker's insertion theorem
\cite{dowker:51}, see Lemma \ref{lemma-choban-sel-v12:1}. Moreover, it
is still valid in the setting of not necessarily connected-valued
mappings, in which case the intermediate mapping $\varphi$ in
\ref{item:choban-sel-v6:2} is only supposed to be u.s.c., see
Proposition~\ref{proposition-choban-sel-v8:1}.\medskip

If $X$ is normal and $\tau$-paracompact,
$\Phi:X\to \mathscr{F}_\mathbf{c}(\J(\tau))$ is l.s.c.,
$\theta:X\to \mathscr{C}(\J(\tau))$ is u.s.c.\ and
$\theta\leqslant\Phi$, then there also exists a continuous mapping
$\varphi:X\to \mathscr{C}_\mathbf{c}(\J(\tau))$ such that
$\theta\leqslant\varphi\leqslant\Phi$, see Theorem
\ref{theorem-choban-sel-v17:1} and Proposition
\ref{proposition-Icsvm-vgg-rev:1}.  In this regard, the author would
like to know, but doesn't, whether this relaxed insertion property may
still characterise the $\tau$-paracompact spaces. The answer is
``Yes'' in the special case of $\tau=\omega$ (Corollary
\ref{corollary-choban-sel-v10:1}). Moreover, this property implies
both countable paracompactness and $\tau$-collectionwise normality,
see Theorem \ref{theorem-choban-sel-v6:2} below.\medskip

For an infinite cardinal number $\tau$, a space $X$ is called
\emph{$\tau$-collectionwise normal} if every discrete family
$\mathscr{F}$ of closed subsets of $X$, with
$|\mathscr{F}| \leq \tau$, admits a discrete family
$\{U_F : F \in \mathscr{F}\}$ of open subsets of $X$ such that
$F \subseteq U_F$ for each $F \in \mathscr{F}$. If, in this
definition, ``discrete'' is changed to ``locally finite'', we get
another important class of spaces. Namely, a space $X$ is called
\emph{$\tau$-expandable} \cite{krajewski:71} if every locally finite
family $\mathscr{F}$ of closed subsets of $X$, with
$|\mathscr{F}| \leq \tau$, admits a locally finite family
$\{ U_F : F \in \mathscr{F}\}$ of open subsets of $X$ such that
$F \subseteq U_F$ for each $F \in \mathscr{F}$. A space $X$ is
collectionwise normal if it $\tau$-collectionwise normal for every
$\tau$, and $X$ is expandable if it is $\tau$-expandable for
every $\tau$. A space $X$ is normal if and only if it is
$\omega$-collectionwise normal. However, for every $\tau$ there exists
a $\tau$-collectionwise normal space which is not
$\tau^+$-collectionwise normal \cite{przymusinski:78a}. Similarly, a
space $X$ is $\omega$-expandable precisely when it is countably
paracompact \cite[Theorem 2.5]{krajewski:71}. Finally, let us
explicitly remark that a normal space $X$ is $\tau$-expandable if and
only if it is countably paracompact and $\tau$-collectionwise normal
\cite{dowker:56,katetov:58}. Our next result deals with the following
characterisation of $\tau$-expandable spaces.

\begin{theorem}
  \label{theorem-choban-sel-v6:2}
  For an infinite cardinal number $\tau$ and a space $X$, the
  following conditions are equivalent\textup{:}
  \begin{enumerate}
  \item\label{item:choban-sel-v6:3} $X$ is normal and
    $\tau$-expandable.
  \item\label{item:choban-sel-v6:4} If
    $\Phi:X\to \mathscr{C}'_\mathbf{c}(\J(\tau))$ is l.s.c.,
    $\theta:X\to \mathscr{C}(\J(\tau))$ is u.s.c.\ and
    $\theta\leqslant \Phi$, then there exists a continuous mapping
    ${\varphi:X\to \mathscr{C}_\mathbf{c}(\J(\tau))}$ such that
    $\theta\leqslant \varphi\leqslant\Phi$.
  \end{enumerate}
\end{theorem}

For a space $Y$ and $k\in\N$, let
$\mathscr{C}_k(Y)=\left\{S\in 2^Y: |S|\leq k\right\}\subseteq
\mathscr{C}(Y)$. Regarding the role of countable paracompactness as a
component of $\tau$-expandable spaces, we will also prove the
following theorem which is complementary to Theorem
\ref{theorem-choban-sel-v6:2}.

\begin{theorem}
  \label{theorem-choban-sel-v6:3}
  For an infinite cardinal number $\tau$ and a space $X$, the
  following conditions are equivalent\textup{:}
  \begin{enumerate}
  \item\label{item:choban-sel-v6:5} $X$ is 
    $\tau$-collectionwise normal. 
  \item\label{item:choban-sel-v6:6} If
    $\Phi:X\to \mathscr{C}'_\mathbf{c}(\J(\tau))$ is l.s.c., $k\in\N$
    and $\theta:X\to \mathscr{C}_k(\J(\tau))$ is u.s.c.\ with
    $\theta\leqslant\Phi$, then there exists a continuous mapping
    ${\varphi:X\to \mathscr{C}_\mathbf{c}(\J(\tau))}$ such that
    $\theta\leqslant \varphi\leqslant\Phi$.
  \end{enumerate}
\end{theorem}

Regarding the role of the metrizable hedgehog, let us remark that the
set-valued insertion property in Theorems
\ref{theorem-choban-sel-v6:1}, \ref{theorem-choban-sel-v6:2} and
\ref{theorem-choban-sel-v6:3} remains valid if $\J(\tau)$ is replaced
by any (connected) complete metric space $(Y, \rho)$ of topological
weight $w(Y)\leq \tau$, and it is required that the family
$\{\Phi(x): x\in X\}$ is uniformly equi-$LC^0$. Furthermore, one can
only require that $Y$ is completely metrizable and
$\{\Phi(x): x\in X\}$ is equi-``locally connected'' in the sense of
Nepomnyashchii \cite{nepomnyashchii:85,nepomnyashchii:86}. Finally,
let us also remark that some of these results we announced in
\cite{gutev-con:90}.\medskip

The paper is organised as follows. In the next section, we use a
general construction
``$\theta\leqslant\Phi\to \mathscr{C}[\theta,\Phi]$'' assigning a
set-valued mapping $\mathscr{C}[\theta,\Phi]:X\to 2^{\mathscr{C}(Y)}$
in the hyperspace $\mathscr{C}(Y)$ corresponding to a pair of mappings
${\theta,\Phi:X\to 2^Y}$ with $\theta\leqslant\Phi$, see
\eqref{eq:choban-sel-v8:1} and \eqref{eq:choban-sel:1}.  In this
setting, each selection $\Psi:X\to 2^{\mathscr{C}(Y)}$ for
$\mathscr{C}[\theta,\Phi]$ can be transformed into an intermediate
mapping $\theta\leqslant\bigcup\Psi\leqslant\Phi$ by taking the union
of the point-images of $\Psi$. Furthermore, the operation
``$\Psi\to \bigcup\Psi$'' preserves properties of semi-continuity. The
prototype of this construction can be found in \cite{gutev:00b}, it is
also implicitly present in \cite{nepomnyashchii:86}. Based on this
construction, we obtain a general result for the existence of
intermediate continuous mappings, see Theorem
\ref{theorem-choban-sel-v17:1}. In fact, this result is a consequence
of a previous result of the author \cite[Theorem 7.1]{gutev:98}, and
works for set-valued mappings whose domain is a Tychonoff
space. However, in contrast to Theorems \ref{theorem-choban-sel-v6:1},
\ref{theorem-choban-sel-v6:2} and \ref{theorem-choban-sel-v6:3}, it
deals with the so called ``metric''-strongly lower semi-factorizable
mappings (abbreviated ``metric''-s.l.s.f.) rather than l.s.c.\
mappings. Section \ref{sec:interm-select-parac} contains the proof of
Theorem \ref{theorem-choban-sel-v6:1}, which is now obtained as a
consequence of the general result in Theorem
\ref{theorem-choban-sel-v17:1}. An interesting element in this proof
is the special case of countably paracompact normal spaces, Lemma
\ref{lemma-choban-sel-v12:1}, which is shown to be equivalent to
Dowker's insertion theorem \cite{dowker:51}. This is subsequently used
in the proof of the general case of $\tau$-paracompact normal
spaces. The section also contains several other related
observations. Theorems \ref{theorem-choban-sel-v6:2} and
\ref{theorem-choban-sel-v6:3} are obtained in a similar way.  Section
\ref{sec:exampl-strongly-lowe} contains the essential preparation to
apply Theorem \ref{theorem-choban-sel-v17:1} in the proofs of these
theorems, see Theorem \ref{theorem-Icsvm-vgg-rev:1}. Finally, Theorems
\ref{theorem-choban-sel-v6:2} and \ref{theorem-choban-sel-v6:3} are
proved in Section \ref{sec:interm-select-coll}. Just like before, an
interesting element in these proofs is a special case --- that of
normal spaces, see Lemma \ref{lemma-choban-sel-v12:2}. It is now
equivalent to the Kat\v{e}tov-Tong insertion theorem, see
\cite{katetov:51,MR0060211,tong:48,MR0050265}. Several related results
are obtained as well.

\section{Continuous Intermediate Mappings}

In this section, $(Y,\rho)$ is a fixed metric space. For
$\varepsilon>0$, we will use $\mathbf{O}_\varepsilon^\rho(p)$ for the
\emph{open $\varepsilon$-ball} centred at a point $p\in Y$; and
$\mathbf{O}_\varepsilon^\rho(S)=\bigcup_{p\in
  S}\mathbf{O}_\varepsilon^\rho(p)$, whenever $S\subseteq Y$.  In what
follows, we will consider the set $\mathscr{C}(Y)$ as a topological
space equipped with the \emph{Hausdorff topology}, i.e.\ the topology
generated by the \emph{Hausdorff distance} $H(\rho)$ on $Y$ associated
to $\rho$. Recall that $H(\rho)$ is defined by
\[
  H(\rho)(S,T)=\inf\big\{\varepsilon>0: S\subseteq
  \mathbf{O}_\varepsilon^\rho(T)\ \text{and}\ T\subseteq
  \mathbf{O}_\varepsilon^\rho(S)\big\},\quad S,T\in \mathscr{C}(Y).
\]
A mapping $\varphi:X\to \mathscr{F}(Y)$ is continuous precisely when
it is continuous with respect to the \emph{Vietoris topology} on
$\mathscr{F}(Y)$, see \cite[Corollary 9.3]{MR0042109}.  However, on
the hyperspace $\mathscr{C}(Y)$, the Hausdorff topology coincides with
the Vietoris one, see \cite[Theorem 3.3]{MR0042109}.  Thus, a mapping
$\varphi:X\to \mathscr{C}(Y)$ is continuous precisely when it is
continuous with respect to the Hausdorff distance $H(\rho)$ on
$\mathscr{C}(Y)$.\medskip

Motivated by the so called s.f.s.c.\ mappings in \cite{gutev:87a}, a
mapping $\Phi:X\to 2^Y$ was said to be \emph{lower semi-factorizable}
with respect to $\rho$, or \emph{$\rho$-l.s.f.}, \cite{gutev:00a} if
for every closed $F\subseteq X$, every $\varepsilon >0$ and every (not
necessarily continuous) single-valued selection $s:F\to Y$ for
$\Phi\uhr F$, there exists a locally finite cozero-set (in $F$)
covering $\mathscr{U}$ of $F$ and a map $\pi:\mathscr{U}\to F$ such
that $|\mathscr{U}|\leq w(Y)$ and
\[
  s(\pi(U))\in
  \mathbf{O}^\rho_\varepsilon\left(\Phi(x)\right),\quad\text{for every
    $x\in U\in \mathscr{U}$}.
\]
The lower semi-factorizable mappings were very successful in
\cite{gutev:00a} to obtain several selection theorems from a common
point of view. Regarding set-valued continuous selections, the
following refined version of these mappings was introduced in
\cite{gutev:98}. A mapping $\Phi:X\to 2^Y$ is \emph{strongly lower
  semi-factorizable} with respect to $\rho$, or
\emph{$\rho$-s.l.s.f.}, \cite{gutev:98} if for every closed
$F\subseteq X$, every $\varepsilon >0$, and every selection
$\sigma:F\to \mathscr{C}(Y)$ for $\Phi\uhr F$, there exists a locally
finite cozero-set (in $F$) covering $\mathscr{U}$ of $F$ and a map
$\pi :\mathscr{U}\to F$ such that $|\mathscr{U}|\leq w(Y)$ and
\[
  \sigma(\pi(U))\subseteq \mathbf{O}_\varepsilon^\rho(\Phi(x)),\quad
  \text{for every $x\in U\in \mathscr{U}$.}
\]
For the proper understanding of these mappings, let us point out the
following relationship between them, it was obtained in
\cite[Corollary 2.2]{gutev:98}.

\begin{proposition}[\cite{gutev:98}] 
  \label{proposition-choban-sel-v13:3}
  For a space $X$, a metric space $(Y,\rho)$ and $\Psi:X\to 2^Y$,
  define a mapping $\mathscr{C}[\Psi]:X\to 2^{\mathscr{C}(Y)}$ by
  \begin{equation}
    \label{eq:choban-sel-v13:1}
    \mathscr{C}[\Psi](x)=\{S\in \mathscr{C}(Y): S\subseteq
    \Psi(x)\},\quad \text{for every $x\in X$.}
  \end{equation}
  Then $\Psi$ is $\rho$-s.l.s.f.\ if and only if\/ $\mathscr{C}[\Psi]$
  is $H(\rho)$-l.s.f. 
\end{proposition}

Here, we will show that one of the main results of \cite{gutev:98}
implies the existence of continuous intermediate mappings between
special pairs of set-valued mappings.  To this end, following
\cite{gutev:00b}, for $\Phi:X\to \mathscr{F}(Y)$ and
$\theta:X\to \mathscr{C}(Y)$ with $\theta\leqslant \Phi$, we will
associate the mappings
$\mathscr{C}[\theta,\Phi],\mathscr{C}_\mathbf{c}[\theta,\Phi]: X\to
2^{\mathscr{C}(Y)}\cup\{\emptyset\}$ defined by
\begin{align}
  \label{eq:choban-sel-v8:1}
  \mathscr{C}[\theta,\Phi](x)
  &=\left\{S\in
    \mathscr{C}(Y): \theta(x)\subseteq 
    S\subseteq \Phi(x)\right\}\quad\text{and}\\ 
  \label{eq:choban-sel:1}
  \mathscr{C}_\mathbf{c}[\theta,\Phi](x)
  & =\left\{S\in
    \mathscr{C}_\mathbf{c}(Y): \theta(x)\subseteq 
    S\subseteq \Phi(x)\right\},\quad
    x\in X.   
\end{align}
Evidently,
$\mathscr{C}[\theta,\Phi]:X\to \mathscr{F}(\mathscr{C}(Y))$. If the
point-images of $\Phi$ are connected and locally path-connected, then
we also have that
$\mathscr{C}_\mathbf{c}[\theta,\Phi]:X\to
\mathscr{F}_\mathbf{c}(\mathscr{C}_\mathbf{c}(Y))$. Indeed, in this
case, it follows from \cite[Lemma 1.3]{curtis:80} that each element of
$\mathscr{C}[\theta,\Phi](x)$ is contained in some element of
$\mathscr{C}_\mathbf{c}[\theta,\Phi](x)$. Therefore,
$\mathscr{C}_\mathbf{c}[\theta,\Phi](x)$ is also nonempty. Since
$\mathscr{C}_\mathbf{c}[\theta,\Phi](x)$ is closed, it is a
\emph{growth hyperspace} in the sense of Curtis
\cite{curtis:78,curtis:80}. Hence, by a result of Kelley \cite[Lemma
2.3]{zbMATH03102237}, see also \cite[Lemma 1.1]{curtis:80},
$\mathscr{C}_\mathbf{c}[\theta,\Phi](x)$ is path-connected as well,
i.e.\
$\mathscr{C}_\mathbf{c}[\theta,\Phi]:X\to
\mathscr{F}_\mathbf{c}(\mathscr{C}_\mathbf{c}(Y))$. In the sequel, we
will freely rely on this fact without any explicit reference. \medskip

A family $\mathscr{S}\subseteq 2^Y$ is \emph{uniformly equi-$LC^0$}
\cite{michael:56b} if for every $\varepsilon>0$ there is
$\delta(\varepsilon)>0$ such that for every $S\in\mathscr{S}$, every
two points $y_0,y_1\in S$ with ${\rho(y_0,y_1)<\delta(\varepsilon)}$,
can be joined by a path in $S$ of diameter$\ <\varepsilon$. Based on 
strongly lower semi-factorizable mappings and uniformly equi-$LC^0$
families of sets, we now have the following general result about
continuous intermediate mappings. 

\begin{theorem}
  \label{theorem-choban-sel-v17:1}
  Let $X$ be a Tychonoff space, $(Y,\rho)$ be a complete metric space
  and ${\Phi: X\to \mathscr{F}_\mathbf{c}(Y)}$ be such that
  $\{\Phi(x): x\in X\}$ is uniformly equi-$LC^0$. If
  ${\theta:X\to \mathscr{C}(Y)}$ is a selection for $\Phi$ such that
  $\mathscr{C}_\mathbf{c}[\theta,\Phi]: X\to
  \mathscr{F}_\mathbf{c}({\mathscr{C}_\mathbf{c}(Y)})$ is
  $H(\rho)$-s.l.s.f, then there exists a continuous mapping
  $\varphi:X\to \mathscr{C}_\mathbf{c}(Y)$ with
  $\theta\leqslant \varphi\leqslant\Phi$.
\end{theorem}

\begin{proof}
  According to \cite[Lemma 5]{nepomnyashchii:86}, see also \cite[Lemma
  1.4]{curtis:80},
  $\left\{\mathscr{C}_\mathbf{c}[\theta,\Phi](x):x\in X\right\}$ is
  uniformly equi-$LC^0$ in $\mathscr{C}_\mathbf{c}(Y)$. Since
  $\mathscr{C}_\mathbf{c}(Y)$ is closed in $\mathscr{C}(Y)$, it is
  also complete with respect to $H(\rho)$. Hence, by \cite[Theorem
  7.1]{gutev:98}, $\mathscr{C}_\mathbf{c}[\theta,\Phi]$ has a
  continuous selection
  $\Psi:X\to \mathscr{C}_\mathbf{c}(\mathscr{C}_\mathbf{c}(Y))$.
  Finally, define $\varphi:X\to 2^Y$ by $\varphi(x)=\bigcup\Psi(x)$,
  $x\in X$.  Then by \cite[Lemma 1.2]{zbMATH03102237}, see also
  \cite[Theorem 2.5]{MR0042109} and \cite[Lemma 2$''$]{MR0367893},
  $\varphi:X\to \mathscr{C}_\mathbf{c}(Y)$. Moreover, $\varphi$ is
  continuous, see \cite[Theorem 5.7]{MR0042109}. Since
  $\theta(x)\subseteq \varphi(x)\subseteq \Phi(x)$, $x\in X$, the
  proof is complete.
\end{proof}

Regarding the proof of Theorem \ref{theorem-choban-sel-v17:1},
let us explicitly remark that the original formulation of
\cite[Theorem 7.1]{gutev:98} is in terms of controlled extensions of
continuous set-valued selections. Moreover, it contains the extra
condition that $Y$ itself must be uniformly equi-$LC^0$. However, one
can always embed $(Y,\rho)$ isometrically into a Banach space and use 
this Banach space instead of $Y$.\medskip

As we will see in Proposition \ref{proposition-Icsvm-vgg-rev:1} of the
next section, Theorem \ref{theorem-choban-sel-v17:1} contains a result
previously obtained by Nepomnyashchii \cite[Theorem
B]{nepomnyashchii:86}, it is in the special case when $X$ is
paracompact, $\Phi$ is l.s.c.\ and $\theta$ is u.s.c. In the same
setting, an alternative proof of Nepomnyashchii's theorem was given by
Michael \cite{michael:04}. In fact, the proof of Theorem
\ref{theorem-choban-sel-v17:1} is essentially the same as that of
Michael. In this regard, let us also remark that in the setting of the
hyperspace $(\mathscr{C}(Y),H(\rho))$, the construction
\eqref{eq:choban-sel-v8:1} was previously used in \cite[Example
3.11]{gutev:00b} where the idea was the same, namely to define a
set-valued mapping in $\mathscr{C}(Y)$ and take a set-valued selection
of this mapping.\medskip

We conclude this section with another general result about continuous
intermediate mappings. It is complementary to Theorem
\ref{theorem-choban-sel-v17:1} and deals with the case when the
domain $X$ has a covering dimension $\dim(X)=0$. The result is valid
for Tychonoff spaces, but for simplicity we state it for normal
spaces. Let us recall that $X$ is a normal space with $\dim(X)=0$
precisely when every two disjoint closed subsets are contained in
disjoint clopen subsets.

\begin{theorem}
  \label{theorem-choban-sel-v17:2}
  Let $X$ be a normal space with $\dim(X)=0$, $(Y,\rho)$ be a complete
  metric space and $\Phi:X\to \mathscr{F}(Y)$. If
  $\theta:X\to \mathscr{C}(Y)$ is a selection for $\Phi$ such that
  $\mathscr{C}[\theta,\Phi]:X\to \mathscr{F}(\mathscr{C}(Y))$ is
  $H(\rho)$-l.s.f., then there exists a continuous mapping
  $\varphi:X\to \mathscr{C}(Y)$ with
  $\theta\leqslant\varphi\leqslant \Phi$.
\end{theorem}

\begin{proof}
  Since $\mathscr{C}[\theta,\Phi]$ is $H(\rho)$-l.s.f., by
  \cite[Theorem 5.1]{gutev:00a} and \cite[Lemma 2.2]{morita:75}, there
  is a metrizable space $Z$ with $\dim(Z)=0$, an l.s.c.\ mapping
  $\Psi:Z\to \mathscr{F}(\mathscr{C}(Y))$ and a continuous map
  $g:X\to Z$ such that
  $\Psi(g(x))\subseteq \mathscr{C}[\theta,\Phi](x)$, for every
  $x\in X$.  Since $Z$ is paracompact, by a result of Michael
  \cite{michael:56}, see also \cite{michael-pixley:80}, the l.s.c.\
  mapping $\Psi$ has a continuous selection
  $\psi:Z\to \mathscr{C}(Y)$. Then the composite mapping
  $\varphi=\psi\circ g:X\to \mathscr{C}(Y)$ is as required.
\end{proof}

\section{Intermediate Selections and Paracompactness}
\label{sec:interm-select-parac}

In this section, we finalise the proof of Theorem
\ref{theorem-choban-sel-v6:1}. In the one direction, it is based on
Theorem \ref{theorem-choban-sel-v17:1}. To prepare for this, let us
state explicitly the following example of strongly lower
semi-factorizable mappings. 

\begin{proposition}
  \label{proposition-Icsvm-vgg-rev:1}
  Let $X$ be normal and $\tau$-paracompact, $(Y,\rho)$ be a metric
  space with $w(Y)\leq\tau$, and $\Phi:X\to \mathscr{F}_\mathbf{c}(Y)$
  be l.s.c.\ such that $\{\Phi(x): x\in X\}$ is uniformly
  equi-$LC^0$. If $\theta:X\to \mathscr{C}(Y)$ is a u.s.c.\ selection
  for $\Phi$, then the associated mapping
  $\mathscr{C}_\mathbf{c}[\theta,\Phi]:X\to
  \mathscr{F}_\mathbf{c}(\mathscr{C}_\mathbf{c}(Y))$ is
  $H(\rho)$-s.l.s.f.
\end{proposition}

\begin{proof}
  By \cite[Lemma 5.3]{michael:04},
  $\mathscr{C}_\mathbf{c}[\theta,\Phi]$ is l.s.c. Then by
  \cite[Example 2.6]{gutev:98}, it is also $H(\rho)$-s.l.s.f.\ because
  $w(\mathscr{C}_\mathbf{c}(Y))\leq\tau$ and $X$ is normal and
  $\tau$-paracompact. 
\end{proof}

The inequality ``$\theta<\varphi\leqslant \Phi$'' in Theorem
\ref{theorem-choban-sel-v6:1} is based on the following general
property of intermediate mappings.

\begin{proposition}
  \label{proposition-choban-sel-v6:1}
  Let $X$ be normal and $\tau$-paracompact, $Y$ be a completely
  metrizable space with $w(Y)\leq\tau$, $\Phi:X\to \mathscr{F}(Y)$ be
  l.s.c.\ and ${\theta:X\to \mathscr{C}(Y)}$ be u.s.c.\ with
  $\theta<\Phi$.  Then there exists a u.s.c.\ mapping
  $\psi:X\to \mathscr{C}(Y)$ with $\theta<\psi\leqslant \Phi$.
\end{proposition}

\begin{proof}
  Let $\mathscr{B}$ be an open base for the topology of $Y$ with
  $|\mathscr{B}|\leq\tau$. Next, for every $B\in \mathscr{B}$ set
  $V_B=\Phi^{-1}(B)\cap
  \theta^{\#}\left(Y\setminus\overline{B}\right)$. In this way, we get
  an open subset $V_B$ of $X$ such that $B\cap \Phi(x)\neq\emptyset$
  and $\overline{B}\cap \theta(x)=\emptyset$ for every $x\in V_B$.
  Hence, $\{V_B:B\in\mathscr{B}\}$ is an open cover of $X$ because
  $\mathscr{B}$ is a base and by condition,
  $\Phi(x)\backslash\theta(x)\neq\emptyset$ for every $x\in X$.  Since
  $X$ is $\tau$-paracompact, it has a locally finite open cover
  $\{U_B:B\in\mathscr{B}\}$ such that $U_B\subseteq V_B$, for every
  $B\in\mathscr{B}$. Since $X$ is also normal, by the Lefschetz lemma
  \cite{MR0007093}, it has a closed cover $\{F_B:B\in\mathscr{B}\}$
  with $F_B\subseteq U_B$ for every $B\in\mathscr{B}$. Finally, for
  every $B\in\mathscr{B}$, define a set-valued mapping
  $\Phi_B:F_B\to\mathscr{F}(Y)$ by
  $\Phi_B(x)=\overline{\Phi(x)\cap B}$, $x\in F_B$. Obviously, each
  $\Phi_B$ is an l.s.c. selection for $\Phi\uhr F_B$ such that
  \begin{equation}
    \label{eq:choban-sel-v6:2}
    \Phi_B(x)\cap \theta(x)=\emptyset,\quad \text{for every $x\in F_B$.}
  \end{equation}
  By a result of Choban \cite[Theorem 11.2]{zbMATH03366111}, see also
  Michael \cite[Theorem 1.1]{MR0109343}, each $\Phi_B$ admits a
  u.s.c.\ selection $\varphi_B:F_B\to\mathscr{C}(Y)$. Since
  $\{F_B:B\in\mathscr{B}\}$ is a locally finite closed cover of $X$,
  we may now define a u.s.c.\ selection $\varphi:X\to \mathscr{C}(Y)$ for
  $\Phi$ by
  $\varphi(x)=\bigcup\left\{\varphi_B(x):x\in F_B\ \text{and}\ B\in
    \mathscr{B}\right\}$, $x\in X$. Then by
  \eqref{eq:choban-sel-v6:2}, we get that
  $\varphi(x)\cap \theta(x)=\emptyset$ for all $x\in X$. Hence, we may
  define the required intermediate u.s.c.\ mapping
  $\psi:X\to \mathscr{C}(Y)$ by $\psi(x)=\theta(x)\cup\varphi(x)$,
  $x\in X$.
\end{proof}

In this section, and what follows, $\J(2)$ is the hedgehog with two
spines. By identifying the two spines with the intervals $[-1,0]$ and
$[0,1]$, it follows from \eqref{eq:choban-sel-v7:1} that $(\J(2),d)$
is isometric to the usual interval $[-1,1]$ when this interval is
equipped with the Euclidean distance $d(s,t)=|s-t|$, $s,t\in
[-1,1]$. This implies the following immediate property of the
connected subsets of $\J(\tau)$.

\begin{proposition}
  \label{proposition-choban-sel-v17:2}
  If $\tau$ is a cardinal number, then the family
  $\mathscr{F}_\mathbf{c}(\J(\tau))$ is uniformly equi-$LC^0$ in
  $(\J(\tau),d)$.
\end{proposition}

A function $\xi:X\to \R$ is \emph{lower} (\emph{upper})
\emph{semi-continuous} if the set
\[
\{x\in X:\xi(x)>r\}\quad \text{(respectively, $\{x\in X:\xi(x)<r\}$)}
\]
is open in $X$ for every $r\in \R$. For functions $\xi,\eta:X\to \R$,
we write that $\xi< \eta$ ($\xi\leq \eta$), if $\xi(x)<\eta(x)$
(respectively, $\xi(x)\leq \eta(x)$), for every $x\in X$.  In these
terms, we have the following set-valued interpretation of Dowker's
insertion theorem \cite{dowker:51}. 

\begin{lemma}
  \label{lemma-choban-sel-v12:1}
  For a space $X$, the following are equivalent\textup{:}
  \begin{enumerate}
  \item\label{item:choban-sel-v7:1} $X$ is countably paracompact and
    normal. 
  \item\label{item:choban-sel-v7:2} If
    $\Phi:X\to \mathscr{F}_\mathbf{c}(\J(2))$ is l.s.c.,
    $\theta:X\to \mathscr{C}(\J(2))$ is u.s.c.\ and $\theta<\Phi$,
    then there exists a continuous mapping
    ${\varphi:X\to \mathscr{C}_\mathbf{c}(\J(2))}$ such that 
    ${\theta<\varphi\leqslant \Phi}$.
  \item\label{item:choban-sel-v7:3} If $\xi:X\to \R$ is upper
    semicontinuous, $\eta:X\to \R$ is lower semicontinuous and
    $\xi<\eta$, then there exists a continuous function $f:X\to \R$
    such that $\xi < f < \eta$.
  \end{enumerate}
\end{lemma}

\begin{proof}
  To show that
  \ref{item:choban-sel-v7:1}$\implies$\ref{item:choban-sel-v7:2},
  suppose that $X$ is countably paracompact and normal,
  $\Phi:X\to \mathscr{F}_\mathbf{c}([-1,1])$ is l.s.c.\ and
  $\theta:X\to \mathscr{C}([-1,1])$ is u.s.c.\ with $\theta<\Phi$. By
  Proposition \ref{proposition-choban-sel-v6:1}, there exists a
  u.s.c.\ mapping $\psi:X\to \mathscr{C}([-1,1])$ with
  $\theta<\psi\leqslant \Phi$. Finally, since $\Phi$ is convex-valued,
  it follows from Proposition \ref{proposition-Icsvm-vgg-rev:1} that
  $\mathscr{C}_\mathbf{c}[\psi,\Phi]:X\to
  \mathscr{F}_\mathbf{c}(\mathscr{C}_\mathbf{c}([-1,1]))$ is
  $H(\rho)$-s.l.s.f.  Therefore, by Theorem
  \ref{theorem-choban-sel-v17:1}, there exists a continuous mapping
  $\varphi:X\to \mathscr{C}_\mathbf{c}([-1,1])$ such that
  $\theta<\psi\leqslant \varphi\leqslant \Phi$.\smallskip

  Suppose that \ref{item:choban-sel-v7:2} holds, and take functions
  $\xi,\eta:X\to \R$ such that $\xi$ is upper semicontinuous, $\eta$
  is lower semicontinuous and $\xi<\eta$. Using the order preserving
  homeomorphism $h(t)=\frac{t}{1+|t|}$, $t\in \R$, of the real line
  $\R$ onto the interval $(-1, 1)$, we may assume that
  $\xi,\eta:X\to (-1,1)$. Define mapping
  $\Phi_\xi,\theta_\xi:X\to \mathscr{C}_\mathbf{c}([-1,1])$ by
  $\Phi_\xi(x)=[-1,\eta(x)]$ and $\theta_\xi(x)=[-1,\xi(x)]$,
  $x\in X$. Then $\theta_\xi<\Phi_\xi$. Moreover, $\Phi_\xi$ is
  l.s.c.\ and $\theta_\xi$ is u.s.c.\ \cite[Example
  1.2$^*$]{MR0077107}, see also \cite[Proposition
  5.4]{gutev:11c}. Hence, by \ref{item:choban-sel-v7:2}, there exists
  a continuous mapping
  $\varphi_\xi:X\to \mathscr{C}_\mathbf{c}([-1,1])$ with
  $\theta_\xi<\varphi_\xi\leqslant \Phi_\xi$. We may now define a
  continuous function $f_\xi:X\to (-1,1)$ by
  $f_\xi(x)=\max\varphi_\xi(x)$, $x\in X$, see e.g.\ \cite[Proposition
  5.4]{gutev:11c}. According to the definitions of $\Phi_\xi$ and
  $\theta_\xi$, and the property of $\varphi_\xi$, it follows that
  $\xi<f_\xi\leq \eta$. Similarly, using the function $\eta$, there
  also exists a continuous function $f_\eta:X\to (-1,1)$ such that
  $\xi\leq f_\eta<\eta$. Then $f=\frac{f_\xi+f_\eta}2$ is as required
  in \ref{item:choban-sel-v7:3}. Since the implication
  \ref{item:choban-sel-v7:3}$\implies$\ref{item:choban-sel-v7:1} is a
  part of Dowker's insertion theorem \cite[Theorem 4]{dowker:51}, the
  proof is complete.
\end{proof}

Regarding the proper place of Lemma \ref{lemma-choban-sel-v12:1}, it
is evident from its proof that the countable paracompactness of $X$ is
only essential for the relation $\theta<\varphi\leqslant \Phi$,
compare with Lemma \ref{lemma-choban-sel-v12:2} in the last section.

\begin{proof}[Proof of Theorem \ref{theorem-choban-sel-v6:1}]
  Let $\tau$ be an infinite cardinal number, $X$ be a
  $\tau$-paracom\-pact normal space,
  $\Phi:X\to \mathscr{F}_\mathbf{c}(\J(\tau))$ be l.s.c.\ and
  $\theta:X\to \mathscr{C}(\J(\tau))$ be u.s.c.\ with $\theta<\Phi$.
  By Proposition \ref{proposition-choban-sel-v6:1}, there exists a
  u.s.c.\ mapping $\psi:X\to \mathscr{C}(\J(\tau))$ such that
  $\theta<\psi\leqslant \Phi$. By Proposition
  \ref{proposition-choban-sel-v17:2}, the family $\{\Phi(x):x\in X\}$
  is uniformly equi-$LC^0$. Hence, by Proposition
  \ref{proposition-Icsvm-vgg-rev:1}, the mapping
  $\mathscr{C}_\mathbf{c}[\psi,\Phi]:X\to
  \mathscr{F}_\mathbf{c}(\mathscr{C}_\mathbf{c}(\J(\tau)))$ is
  $H(\rho)$-s.l.s.f. Thus, according to Theorem
  \ref{theorem-choban-sel-v17:1}, there exists a continuous mapping
  $\varphi:X\to \mathscr{C}_\mathbf{c}(Y)$ such that
  $\theta<\psi\leqslant \varphi\leqslant \Phi$. This shows
  \ref{item:choban-sel-v6:2} of Theorem \ref{theorem-choban-sel-v6:1}.
  \smallskip

  Conversely, assume that $X$ is as in \ref{item:choban-sel-v6:2} of
  Theorem \ref{theorem-choban-sel-v6:1}. Since
  ${\J(2)\subseteq \J(\tau)}$, it follows from Lemma
  \ref{lemma-choban-sel-v12:1} that $X$ is countably paracompact and
  normal. To see that $X$ is also $\tau$-paracompact, take an open
  cover $\mathscr{U}$ of $X$ with $|\mathscr{U}|\leq \tau$. Next, for
  every $x\in X$, let $\mathscr{U}_x=\{U\in \mathscr{U}: x\in
  U\}$. When $\mathscr{U}$ is equipped with the discrete topology,
  this defines an l.s.c.\ mapping $x\to \mathscr{U}_x$, $x\in X$, see
  the proof of \cite[Theorem 11.2]{zbMATH03366111}. In our case, we
  may use the corresponding sub-hedgehog $\J(\mathscr{U}_x)$ to get an
  l.s.c.\ mapping into the subsets of $\J(\mathscr{U})$ corresponding
  to the cover $\mathscr{U}$. Namely, define
  $\Phi:X\to \mathscr{F}_\mathbf{c}(\J(\mathscr{U}))$ by
  $\Phi(x)=\J(\mathscr{U}_x)$, $x\in X$. Then $\Phi$ is l.s.c.\ which
  follows easily from the fact that
  $\Phi(x)\setminus \{0\}=(0,1]\times \mathscr{U}_x$, $x\in
  X$. Moreover, $\theta(x)=\{0\}$, $x\in X$, is clearly a u.s.c.\
  selection for $\Phi$ with $\Phi(x)\setminus \theta(x)\neq \emptyset$
  for each $x\in X$. Hence, by assumption, there exists a continuous
  mapping $\varphi:X\to \mathscr{C}_\mathbf{c}(\J(\mathscr{U}))$ such
  that $\theta<\varphi\leqslant \Phi$. For convenience, for each
  $U\in \mathscr{U}$ and $n\in \N$, set
  $O[U,n]=\left(2^{-n},1\right]\times\{U\}$ which is an open subset of
  $\J(\mathscr{U})$. Then for each $n\in \N$, it follows from
  \eqref{eq:choban-sel-v7:1} that the family
  $\Omega_n=\{O[U,n]: U\in \mathscr{U}\}$ is discrete in
  $\J(\mathscr{U})$. Accordingly, the family
  \[
    \mathscr{V}_n=\varphi^{-1}(\Omega_n)=\left\{\varphi^{-1}(O[U,n]): U\in
      \mathscr{U}\right\}
  \]
  is locally finite and open in $X$ because the mapping $\varphi$ is
  compact-valued and both l.s.c.\ and u.s.c. Moreover, $\mathscr{V}_n$
  refines $\mathscr{U}$ because $\varphi(x)\cap O[U,n]\neq \emptyset$
  implies $\Phi(x)\cap O[U,n]\neq \emptyset$ and, therefore, $x\in
  U$. Finally, $\bigcup_{n\in\N}\mathscr{V}_n$ is a cover of $X$
  because
  $\varphi(x)\setminus \{0\}=\varphi(x)\setminus \theta(x)\neq
  \emptyset$ for all $x\in X$. This shows that $X$ is
  $\tau$-paracompact being countably paracompact and normal, see
  \cite[Theorem 1.1]{morita:60}.
\end{proof}

We conclude this section with some related observations. 

\begin{proposition}
  \label{proposition-choban-sel-v8:1}
  For an infinite cardinal number $\tau$ and a space $X$, the
  following conditions are equivalent\textup{:}
  \begin{enumerate}
  \item\label{item:choban-sel-v8:1} $X$ is normal and
    $\tau$-paracompact. 
  \item \label{item:choban-sel-v8:2} If $Y$ is a completely metrizable
    space with $w(Y)\leq\tau$, $\Phi:X\to \mathscr{F}(Y)$ is l.s.c.,
    $\theta:X\to \mathscr{C}(Y)$ is u.s.c.\ and $\theta<\Phi$,
    then there exists a u.s.c.\ mapping $\psi:X\to \mathscr{C}(Y)$
    with $\theta<\psi\leqslant \Phi$.
  \end{enumerate}  
\end{proposition}

\begin{proof}
  The implication
  \ref{item:choban-sel-v8:1}$\implies$\ref{item:choban-sel-v8:2} is
  Proposition \ref{proposition-choban-sel-v6:1}. Conversely, assume
  that \ref{item:choban-sel-v8:2} holds and take an l.s.c.\ mapping
  $\Phi:X\to \mathscr{F}(Y)$, where $Y$ is a completely metrizable
  space with $w(Y)\leq\tau$. Next, let $Y^*$ be the space obtained
  from $Y$ by adding an isolated point $*$ to $Y$. We may now define
  another l.s.c.\ mapping $\Phi^*:X\to \mathscr{F}(Y^*)$ by
  $\Phi^*(x)=\Phi(x)\cup\{*\}$, $x\in X$. Also, let
  $\theta:X\to \mathscr{C}(Y^*)$ be the constant mapping
  $\theta(x)=\{*\}$, $x\in X$. Then $\theta$ is a u.s.c.\ selection
  for $\Phi^*$ with
  $\Phi^*(x)\setminus \theta(x)=\Phi(x)\neq \emptyset$, $x\in
  X$. Hence, by \ref{item:choban-sel-v8:2}, $\Phi^*$ admits a u.s.c.\
  selection $\psi:X\to \mathscr{C}(Y^*)$ such that
  $\varphi(x)=\psi(x)\setminus\theta(x)\neq\emptyset$ for every
  $x\in X$. Evidently, this defines a u.s.c.\ selection
  $\varphi:X\to \mathscr{C}(Y)$ for $\Phi$. According to \cite[Theorem
  11.2]{zbMATH03366111}, $X$ is normal and $\tau$-paracompact. 
\end{proof}

Proposition \ref{proposition-choban-sel-v8:1} is complementary to
similar characterisations of $\tau$-expandable spaces and
$\tau$-collectionwise normal spaces obtained in
\cite{miyazaki:01a}. Regarding this proposition, let us also remark
that upper semi-continuity of the intermediate selection $\psi$ cannot
be strengthened to continuity. In fact, in this case, continuity of
the intermediate selection is equivalent to $X$ having a covering
dimension zero.

\begin{proposition}
  \label{proposition-choban-sel-v8:2}
  For an infinite cardinal number $\tau$ and a space $X$, the
  following conditions are equivalent\textup{:}
  \begin{enumerate}
  \item\label{item:choban-sel-v8:3} $X$ is a 
    $\tau$-paracompact normal space with $\dim(X)=0$. 
  \item\label{item:choban-sel-v8:4} If $Y$ is a completely metrizable
    space with $w(Y)\leq\tau$, $\Phi:X\to \mathscr{F}(Y)$ is l.s.c.,
    $\theta:X\to \mathscr{C}(Y)$ is u.s.c.\ and $\theta<\Phi$, then
    there exists a continuous mapping $\varphi:X\to \mathscr{C}(Y)$
    with $\theta<\varphi\leqslant\Phi$.
  \end{enumerate}  
\end{proposition}

\begin{proof}
  To show that
  \ref{item:choban-sel-v8:3}$\implies$\ref{item:choban-sel-v8:4}, let
  $Y$ be a completely metrizable space with $w(Y)\leq \tau$,
  $\Phi:X\to \mathscr{F}(Y)$ be l.s.c.\ and
  $\theta:X\to \mathscr{C}(Y)$ be u.s.c.\ with $\theta<\Phi$. By
  Proposition \ref{proposition-choban-sel-v6:1}, there exists a
  u.s.c.\ mapping $\psi:X\to \mathscr{C}(Y)$ such that
  $\theta<\psi\leqslant \Phi$. Next, consider the mapping
  $\mathscr{C}[\psi,\Phi]:X\to \mathscr{F}(\mathscr{C}(Y))$ which is
  l.s.c., see \cite[Example 3.11]{gutev:00b}. Since
  $w(\mathscr{C}(Y))\leq \tau$, it follows from \cite[Theorem
  11.1]{zbMATH03366111}, see also \cite[Theorem 2]{michael:56}, that
  $\mathscr{C}[\psi,\Phi]$ has a continuous selection
  $\varphi:X\to \mathscr{C}(Y)$. According to
  \eqref{eq:choban-sel-v8:1}, this $\varphi$ is as required. To show
  the converse, assume that \ref{item:choban-sel-v8:4} holds. Then by
  Proposition \ref{proposition-choban-sel-v8:1}, $X$ is normal and
  $\tau$-paracompact. To show that $\dim(X)=0$, take disjoint closed
  sets $F_1,F_2\subseteq X$ and, for convenience, set $Y=\{0,1,2\}$.
  Next, define an l.s.c.\ mapping $\Phi:X\to 2^Y$ by $\Phi(x)=\{0,i\}$
  if $x\in F_i$, and $\Phi(x)=Y$ otherwise. Also, just like before,
  consider the constant mapping $\theta(x)=\{0\}$, $x\in X$. Then
  $\theta<\Phi$ and by \ref{item:choban-sel-v8:4}, there exists a
  continuous mapping $\varphi:X\to 2^Y$ with
  $\theta<\varphi\leqslant \Phi$. In particular,
  $\varphi(x)\neq \{0\}$, $x\in X$, and we may define the disjoint
  sets $U_i=\varphi^{\#}(\{0,i\})$, $i=1,2$. Since $Y$ is discrete and
  $\varphi$ is continuous, $U_1$ and $U_2$ are clopen sets. Moreover,
  by the definition of $\Phi$, we have that $F_i\subseteq U_i$,
  $i=1,2$. This shows that $\dim(X)=0$.
\end{proof}

In setting of paracompact spaces, the implication
\ref{item:choban-sel-v8:3}$\implies$~\ref{item:choban-sel-v8:4} of
Proposition \ref{proposition-choban-sel-v8:2} was stated by
Nepomnyashchii \cite{nepomnyashchii:86}.

\section{Collectionwise Normality and ``Metric''-s.l.s.f.  Mappings}
\label{sec:exampl-strongly-lowe}

The proofs of Theorems \ref{theorem-choban-sel-v6:2} and
\ref{theorem-choban-sel-v6:3} are similar to that of Theorem
\ref{theorem-choban-sel-v6:1} being based on a reduction to Theorem
\ref{theorem-choban-sel-v17:1}. However, in the setting of
$\mathscr{C}'(Y)$-valued mappings, there is no simple version of
Proposition \ref{proposition-Icsvm-vgg-rev:1}. The reason is that for
$\Phi:X\to \mathscr{C}'(Y)$ and a selection
$\theta:X\to \mathscr{C}(Y)$ for $\Phi$, the associated mapping
$\mathscr{C}[\theta,\Phi]$ in \eqref{eq:choban-sel-v8:1} is not
necessarily $\mathscr{C}'(\mathscr{C}(Y))$-valued; similarly, for the
construction in \eqref{eq:choban-sel:1}. This requires an extra
property of the selection $\theta:X\to \mathscr{C}(Y)$.

\begin{theorem}
  \label{theorem-Icsvm-vgg-rev:1}
  Let $X$ be $\tau$-collectionwise normal, $(Y,\rho)$ be a connected
  metric space with $w(Y)\leq\tau$, and
  $\Phi:X\to \mathscr{C}'_\mathbf{c}(Y)$ be l.s.c.\ such that
  $\{\Phi(x): x\in X\}$ is uniformly equi-$LC^0$. Also, let $Z$ be a
  metrizable space, $g:X\to Z$ be continuous and
  $\Theta:Z\to \mathscr{C}(Y)$ be u.s.c.\ such that
  $\psi=\Theta\circ g\leqslant \Phi$.  Then the mapping
  $\mathscr{C}_\mathbf{c}[\psi,\Phi]:X\to
  \mathscr{F}_\mathbf{c}(\mathscr{C}_\mathbf{c}(Y))$ is
  $H(\rho)$-s.l.s.f.
\end{theorem}

The proof of Theorem \ref{theorem-Icsvm-vgg-rev:1} is separated into
two parts --- the case of the constant mapping $\Phi(x)=Y$, $x\in X$,
and that of $\Phi(x)$ being compact for each $x\in X$.

\begin{proposition}
  \label{proposition-choban-sel-v4:3}
  Under the conditions of Theorem \ref{theorem-Icsvm-vgg-rev:1},
  suppose that $\Phi(x)=Y$ for every $x\in X$.  Then
  $\mathscr{C}_\mathbf{c}[\psi,\Phi]:X\to
  \mathscr{F}_\mathbf{c}(\mathscr{C}_\mathbf{c}(Y))$ is
  $H(\rho)$-s.l.s.f.
\end{proposition}

\begin{proof}
  Let $Z$, $\Theta:Z\to \mathscr{C}(Y)$, $g:X\to Z$ and
  $\psi=\Theta\circ g:X\to \mathscr{C}(Y)$ be as in Theorem
  \ref{theorem-Icsvm-vgg-rev:1}. Set $\Psi(z)=Y$, $z\in Z$, and
  consider the associated mapping
  $\mathscr{C}_\mathbf{c}[\Theta,\Psi]$ defined as in
  \eqref{eq:choban-sel:1}. If $F\subseteq X$, then $g(F)$ is
  paracompact being metrizable. Hence, by Proposition
  \ref{proposition-Icsvm-vgg-rev:1},
  $\mathscr{C}_\mathbf{c}[\Theta,\Psi]\uhr g(F)$ is
  $H(\rho)$-s.l.s.f.\ for every $F\subseteq X$. Since
  $\mathscr{C}_\mathbf{c}[\psi,\Phi](x)=\mathscr{C}_\mathbf{c}[\Theta,\Psi](g(x))$
  for every $x\in X$, this implies that
  $\mathscr{C}_\mathbf{c}[\psi,\Phi]$ is also $H(\rho)$-s.l.s.f.
\end{proof}

It will be useful to state the case of compact-valued mappings in a
little bit more general setting. To this end, let us recall that for
an infinite cardinal number $\tau$, a space $X$ is called
\emph{$\tau$-PF-normal} (see \cite{smith:72}) if every point-finite
open cover of $X$ of cardinality $\leq \tau$ is normal.  Every
$\tau$-collectionwise normal space is $\tau$-PF-normal
\cite{michael:55} (see also Kand\^o \cite{MR0063648} and Nedev
\cite{MR644284}), and $\omega$-PF-normality coincides with normality
\cite{morita:48}.  However, PF-normality is neither identical to
collectionwise normality (see Bing's example \cite{bing:51} and
\cite[Example 1]{michael:55}), nor to normality (\cite[Example
1]{michael:55}). For some properties of PF-normal spaces, the
interested reader is referred to \cite[Section 3]{MR1961298} and
\cite{michael:55}.

\begin{proposition}
  \label{proposition-Icsvm-vgg-rev:2}
  Let $X$ be $\tau$-PF-normal, $(Y,\rho)$ be a metric space with
  $w(Y)\leq\tau$, and $\Phi:X\to \mathscr{C}_\mathbf{c}(Y)$ be l.s.c.\
  such that $\{\Phi(x): x\in X\}$ is uniformly equi-$LC^0$. If
  $\psi:X\to \mathscr{C}(Y)$ is a u.s.c.\ selection for $\Phi$, then
  $\mathscr{C}_\mathbf{c}[\psi,\Phi]:X\to
  \mathscr{F}_\mathbf{c}(\mathscr{C}_\mathbf{c}(Y))$ is
  $H(\rho)$-s.l.s.f.
\end{proposition}

\begin{proof}
  Set $E=\mathscr{C}_\mathbf{c}(Y)$ and $d=H(\rho)$. Then $(E,d)$ is a
  metric space with ${w(E)\leq \tau}$. Moreover, by \cite[Lemma
  5.3]{michael:04},
  $\Psi=\mathscr{C}_\mathbf{c}[\psi,\Phi]:X\to
  \mathscr{C}_\mathbf{c}(E)$ is l.s.c. Consider the mapping
  $\mathscr{C}[\Psi]$ defined as in \eqref{eq:choban-sel-v13:1}, which
  remains l.s.c., see \cite[Lemma 3.4$'$]{nepomnyashchii:85}. In
  Nedev's terminology \cite{MR644284}, $X$ is $\tau$-PF-normal
  precisely when it is $\tau^+$-pointwise-$\aleph_0$-paracompact
  space. Hence, by \cite[Lemma 3.5]{MR644284}, $\mathscr{C}[\Psi]$ has
  the \emph{Selection Factorisation Property} in the sense of
  \cite{choban-nedev:74,MR644284} because $w(C(E))\leq
  \tau<\tau^+$. According to the proof of \cite[Example
  4.2]{gutev:00a}, this implies that $\mathscr{C}[\Psi]$ is
  $H(d)$-l.s.f. Therefore, by Proposition
  \ref{proposition-choban-sel-v13:3},
  $\Psi=\mathscr{C}_\mathbf{c}[\psi,\Phi]$ is $d=H(\rho)$-s.l.s.f.
\end{proof}

Finally, we also need the following property of $H(\rho)$-s.l.s.f.\
mappings defined on collectionwise normal spaces. 

\begin{proposition}
  \label{proposition-choban-sel-v5:2}
  Let $X$ be $\tau$-collectionwise normal, $(E,d)$ be a metric
  space and $\Psi:X\to 2^E$ be an l.s.c.\ mapping. Suppose that
  $F\subseteq X$ is a closed set, $\varepsilon>0$, $\mathscr{V}$ is a
  point-finite open in $F$ cover of $F$ with $|\mathscr{V}|\leq\tau$,
  and $K_V\in \mathscr{C}(E)$, $V\in\mathscr{V}$, are compact sets
  such that $K_V\subseteq \mathbf{O}_\varepsilon^d(\Psi(x))$, for
  every $x\in V$. Then there exists a locally finite open in $X$ cover
  $\left\{U_V: V\in\mathscr{V}\right\}$ of $F$ such that
  \[
    U_V\cap F\subseteq V\quad\text{and}\quad K_V\subseteq
    \mathbf{O}_\varepsilon^d(\Psi(x)),\quad \text{for every
      $x\in U_V$ and $V\in\mathscr{V}$.}
  \]
\end{proposition}

\begin{proof}
  Take $V\in \mathscr{V}$. Since $K_V$ is compact and $\Psi$ is
  l.s.c., by \cite[Lemma 11.3]{michael:56b}, there exists an open set
  $W_V\subseteq X$ such that $W_V\cap F=V$ and
  $K_V\subseteq \mathbf{O}_\varepsilon^d(\Psi(x))$, for every
  $x\in W_V$. Thus, we get an open in $X$ and point-finite in $F$
  cover $\{W_V:V\in \mathscr{V}\}$ of $F$. Since $X$ is
  $\tau$-collectionwise normal and $|\mathscr{V}|\leq \tau$, by
  \cite[Lemma 1]{MR0418033}, see also \cite[Lemma 1.6]{MR644284},
  there exists a locally finite open in $X$ cover
  $\{U_V:V\in \mathscr{V}\}$ of $F$ with $U_V\subseteq W_V$,
  $V\in \mathscr{V}$.
\end{proof}

\begin{proof}[Proof of Theorem \ref{theorem-Icsvm-vgg-rev:1}]
  Let $X$, $(Y,\rho)$, $\Phi$ and $\psi$ be as in that theorem. Take
  $\varepsilon>0$, a closed set $F\subseteq X$ and a selection
  $\sigma:F\to \mathscr{C}(\mathscr{C}_\mathbf{c}(Y))$ for
  $\mathscr{C}_\mathbf{c}[\psi,\Phi]\uhr F$. Set $\Psi(x)=Y$,
  $x\in X$.  Then by Proposition \ref{proposition-choban-sel-v4:3},
  $\mathscr{C}_\mathbf{c}[\psi,\Psi]$ is $H(\rho)$-s.l.s.f. Hence, by
  Proposition \ref{proposition-choban-sel-v5:2} applied with
  $E=\mathscr{C}(Y)$ and $d=H(\rho)$, there exists a locally finite
  open in $X$ cover $\mathscr{V}_0$ of $F$ and a map
  $\pi_0:\mathscr{V}_0\to F$ such that $|\mathscr{V}_0|\leq \tau$ and
  \begin{equation}
    \label{eq:choban-sel-v5:1}
    \sigma(\pi_0(V))\subseteq
    \mathbf{O}_\varepsilon^{H(\rho)}(\mathscr{C}_\mathbf{c}[\psi,\Psi](x)), \quad
    \text{for every $x\in V\in \mathscr{V}_0$.}
  \end{equation}
  Next, as in the proof of Proposition
  \ref{proposition-choban-sel-v5:2}, for each $V\in\mathscr{V}$,
  define an open subset $U_V$ of $V$ by
  $U_V=\left\{x\in V: \sigma(\pi_0(V))\subseteq
    \mathbf{O}_\varepsilon^{H(\rho)}(\mathscr{C}_\mathbf{c}[\psi,\Phi](x))\right\}$. Thus,
  we get a locally finite open in $X$ family
  $\mathscr{U}_0=\{U_V:V\in \mathscr{V}_0\}$ such that
  \begin{equation}
    \label{eq:choban-sel-v5:2}
    \sigma(\pi_0(V))\subseteq
    \mathbf{O}_\varepsilon^{H(\rho)}(\mathscr{C}_\mathbf{c}[\psi,\Phi](x)), \quad
    \text{for every $x\in U_V$ and $V\in \mathscr{V}_0$.}
  \end{equation}
  Moreover, if $\Phi(x)=Y$ for some $x\in V\in \mathscr{V}_0$, then
  $\mathscr{C}_\mathbf{c}[\psi,\Phi](x)=\mathscr{C}_\mathbf{c}[\psi,\Psi](x)$
  and by \eqref{eq:choban-sel-v5:1}, $x\in U_V$. Therefore, $\Phi$ is
  compact-valued on the closed set
  $F_1=F\setminus \bigcup\mathscr{U}_0$. Hence, by Propositions
  \ref{proposition-Icsvm-vgg-rev:2} and
  \ref{proposition-choban-sel-v5:2}, there is a locally finite open in
  $X$ cover $\mathscr{U}_1$ of $F_1$ and a map
  $\pi_1:\mathscr{U}_1\to F_1$ such that $|\mathscr{U}_1|\leq \tau$
  and
  \begin{equation}
    \label{eq:choban-sel-v5:3}
    \sigma(\pi_1(U))\subseteq
    \mathbf{O}_\varepsilon^{H(\rho)}(\mathscr{C}_\mathbf{c}[\psi,\Phi](x)), \quad
    \text{for every $x\in U\in \mathscr{U}_1$.}
  \end{equation}
  We can now take $\mathscr{U}$ to be the disjoint union of
  $\mathscr{U}_0$ and $\mathscr{U}_1$, and $\pi:\mathscr{U}\to F$ to
  be defined by $\pi\uhr \mathscr{U}_i=\pi_i$, $i=0,1$. Evidently, by
  \eqref{eq:choban-sel-v5:2} and \eqref{eq:choban-sel-v5:3}, this
  implies that $\mathscr{C}_\mathbf{c}[\psi,\Phi]$ is
  $H(\rho)$-s.l.s.f.
\end{proof}

Let us explicitly remark that in Theorem
\ref{theorem-Icsvm-vgg-rev:1}, the point-images of $\Phi$ were
required to be connected and $\{\Phi(x):x\in X\}$ to be uniformly
equi-$LC^0$ only to make sure that
$\mathscr{C}_\mathbf{c}[\psi,\Phi]:X\to
\mathscr{F}_\mathbf{c}(\mathscr{C}_\mathbf{c}(Y))$. However, this
played no other role in the proof Theorem
\ref{theorem-Icsvm-vgg-rev:1}. If $Y$ is a metrizable space,
$\Phi:X\to \mathscr{F}(Y)$ and $\theta:X\to \mathscr{C}(Y)$ is a
u.s.c.\ selection for $\Phi$, then the mapping
$\mathscr{C}[\theta,\Phi]$ defined as in \eqref{eq:choban-sel-v8:1},
takes values in $ \mathscr{F}(\mathscr{C}(Y))$ and is l.s.c., see the
proof of \cite[Example 3.11]{gutev:00b}. Hence, we have the following
further result from the proof of Theorem
\ref{theorem-Icsvm-vgg-rev:1}.

\begin{theorem}
  \label{theorem-Icsvm-vgg-rev:2}
  Let $X$ be $\tau$-collectionwise normal, $(Y,\rho)$ be a metric
  space with $w(Y)\leq\tau$, and $\Phi:X\to \mathscr{C}'(Y)$ be
  l.s.c. Also, let $Z$ be a metrizable space, $g:X\to Z$ be continuous
  and $\Theta:Z\to \mathscr{C}(Y)$ be u.s.c.\ such that
  $\psi=\Theta\circ g\leqslant \Phi$.  Then
  $\mathscr{C}[\psi,\Phi]:X\to \mathscr{F}(\mathscr{C}(Y))$ is
  $H(\rho)$-s.l.s.f.\ and, in particular, also $H(\rho)$-l.s.f.
\end{theorem}

\section{Intermediate Selections and Collectionwise Normality}
\label{sec:interm-select-coll}

A mapping $\theta: X \to 2^Y$ has the \emph{locally finite lifting
  property} \cite[(3.3)]{gutev:00b} (see also
\cite{MR1961298,zbMATH06047968,miyazaki:01a}) if for
every locally finite family $\mathscr{F}$ of closed subsets of $Y$,
there is a locally finite family $\{ U_F: F \in \mathscr{F}\}$ of open
subsets of $X$ such that $\theta^{-1}(F) \subseteq U_F$ for each
$F \in \mathscr{F}$. The following key observation will be used in
the proofs of both Theorem \ref{theorem-choban-sel-v6:2} and Theorem
\ref{theorem-choban-sel-v6:3}.

\begin{proposition}
  \label{proposition-choban-sel-v9:1}
  Let $X$ be $\tau$-collectionwise normal, $(Y,\rho)$ be a connected
  complete metric space with $w(Y)\leq\tau$, and
  $\Phi:X\to \mathscr{C}'_\mathbf{c}(Y)$ be l.s.c.\ such that
  $\{\Phi(x): x\in X\}$ is uniformly equi-$LC^0$. Also, let
  $\theta:X\to \mathscr{C}(Y)$ be a u.s.c.\ selection for $\Phi$ which
  has the locally finite lifting property. Then there exists a
  continuous mapping $\varphi:X\to \mathscr{C}_\mathbf{c}(Y)$ with
  $\theta\leqslant \varphi\leqslant \Phi$.
\end{proposition}

\begin{proof}
  Since $\theta:X\to \mathscr{C}(Y)$ has the locally finite lifting
  property, by \cite[Theorem 3.1]{zbMATH06047968}, there exists a
  metrizable space $Z$, a continuous map $g:X\to Z$ and a u.s.c.\
  mapping ${\Theta:Z\to \mathscr{C}(Y)}$ such that
  $\theta\leqslant \psi=\Theta\circ g\leqslant \Phi$. Accordingly, by
  Theorem \ref{theorem-Icsvm-vgg-rev:1}, the mapping
  $\mathscr{C}_\mathbf{c}[\psi,\Phi]:X\to
  \mathscr{F}_\mathbf{c}(\mathscr{C}_\mathbf{c}(Y))$ is
  $H(\rho)$-s.l.s.f. Thus, by Theorem \ref{theorem-choban-sel-v17:1},
  there exists a continuous mapping
  $\varphi:X\to \mathscr{C}_\mathbf{c}(Y)$ such that
  $\theta\leqslant \psi \leqslant\varphi\leqslant \Phi$.
\end{proof}

The other key element in these proofs  is
the following set-valued interpretation of the Kat\v{e}tov-Tong
insertion theorem, see Tong \cite{tong:48,MR0050265} and Kat\v{e}tov
\cite{katetov:51,MR0060211}. 

\begin{lemma}
  \label{lemma-choban-sel-v12:2}
    For a space $X$, the following are equivalent\textup{:}
  \begin{enumerate}
  \item\label{item:choban-sel-v11:2} $X$ is normal.
  \item\label{item:choban-sel-v11:3} If
    $\Phi:X\to \mathscr{C}_\mathbf{c}(\J(2))$ is l.s.c.,
    $\theta:X\to \mathscr{C}(\J(2))$ is u.s.c.\ and
    $\theta\leqslant \Phi$, then there exists a continuous mapping
    ${\varphi:X\to \mathscr{C}_\mathbf{c}(\J(2))}$ such that 
    $\theta\leqslant \varphi\leqslant \Phi$.
  \item\label{item:choban-sel-v11:4} If $\xi:X\to \R$ is upper
    semicontinuous, $\eta:X\to \R$ is lower semicontinuous and
    $\xi\leq\eta$, then there exists a continuous function $f:X\to \R$
    such that $\xi \leq f \leq \eta$.
  \end{enumerate}
\end{lemma}

\begin{proof}
  As in the proof of Lemma \ref{lemma-choban-sel-v12:1}, we identify
  $\J(2)$ with the interval $[-1,1]$. To show that
  \ref{item:choban-sel-v11:2}$\implies$\ref{item:choban-sel-v11:3},
  suppose that $X$ is normal,
  $\Phi:X\to \mathscr{C}_\mathbf{c}([-1,1])$ is l.s.c.\ and
  $\theta:X\to \mathscr{C}([-1,1])$ is a u.s.c.\ selection for
  $\Phi$. Since $\Phi$ is compact-convex-valued, by Proposition
  \ref{proposition-Icsvm-vgg-rev:2}, the associated mapping
  $\mathscr{C}_\mathbf{c}[\theta,\Phi]$ is $H(d)$-s.l.s.f. Hence, just
  like before, the required
  $\varphi:X\to \mathscr{C}_\mathbf{c}([-1,1])$ is given by Theorem
  \ref{theorem-choban-sel-v17:1}. \smallskip

  To see that
  \ref{item:choban-sel-v11:3}$\implies$\ref{item:choban-sel-v11:4},
  following the proof of Lemma \ref{lemma-choban-sel-v12:1}, take
  functions $\xi,\eta:X\to (-1,1)$ such that $\xi$ is upper
  semicontinuous, $\eta$ is lower semicontinuous and
  $\xi\leq\eta$. Next, define mappings
  $\Phi,\theta:X\to \mathscr{C}_\mathbf{c}([-1,1])$ by
  $\Phi(x)=[-1,\eta(x)]$ and $\theta(x)=[-1,\xi(x)]$, $x\in X$. Then
  $\Phi$ is l.s.c.\ and $\theta$ is a u.s.c.\ selection for
  $\Phi$. Hence, by \ref{item:choban-sel-v11:3}, there exists a
  continuous mapping $\varphi:X\to \mathscr{C}_\mathbf{c}([-1,1])$
  such that $\theta\leqslant \varphi \leqslant \Phi$. The function
  $f:X\to (-1,1)$, defined by $f(x)=\max\varphi(x)$, $x\in X$, is as
  required in \ref{item:choban-sel-v11:4}. Since the implication
  \ref{item:choban-sel-v11:4}$\implies$\ref{item:choban-sel-v11:2} is
  a part of the Kat\v{e}tov-Tong insertion theorem, the proof is
  complete.
\end{proof}

Regarding the implication
\ref{item:choban-sel-v11:2}$\implies$\ref{item:choban-sel-v11:3} of
Lemma \ref{lemma-choban-sel-v12:2}, let us remark that we used Theorem
\ref{theorem-choban-sel-v17:1} which is behind the framework of all
results of this paper. However, an alternative proof follows using the
Kat\v{e}tov-Tong insertion theorem.

\begin{proof}[Proof of Theorem \ref{theorem-choban-sel-v6:2}]
  Let $X$ be normal and $\tau$-expandable,
  ${\Phi:X\to \mathscr{C}_\mathbf{c}'(\J(\tau))}$ be l.s.c.\ and
  $\theta:X\to \mathscr{C}(\J(\tau))$ be a u.s.c.\ selection for
  $\Phi$. By \cite[Example 3.9]{gutev:00b}, $\theta$ has the
  locally finite lifting property. Moreover, by Proposition
  \ref{proposition-choban-sel-v17:2}, $\{\Phi(x):x\in X\}$ is
  uniformly equi-$LC^0$. Hence, the required intermediate mapping
  $\varphi$ is now given by Proposition
  \ref{proposition-choban-sel-v9:1}. This shows the implication
  \ref{item:choban-sel-v6:3}$\implies$\ref{item:choban-sel-v6:4} in
  Theorem \ref{theorem-choban-sel-v6:2}.\smallskip

  To see the inverse implication, assume that $X$ is as in
  \ref{item:choban-sel-v6:4} of Theorem
  \ref{theorem-choban-sel-v6:2}. Since ${\J(2)\subseteq \J(\tau)}$, it
  follows from Lemma \ref{lemma-choban-sel-v12:2} that $X$ is
  normal. Take a locally finite family $\mathscr{F}$ of closed subsets
  of $X$ with $|\mathscr{F}| \leq \tau$. By adding $X$ to
  $\mathscr{F}$, if necessary, we may assume that $\mathscr{F}$ is a
  cover of $X$. Next, define a u.s.c.\ mapping
  $\theta:X\to \mathscr{C}(\J(\mathscr{F}))$ by
  $\theta(x)=\{\langle 1,F\rangle: x\in F\ \text{and}\ F\in
  \mathscr{F}\}$, $x\in X$. Also, let $\Phi(x)=\J(\mathscr{F})$,
  $x\in X$, be the constant mapping. Evidently,
  $\Phi:X\to \mathscr{C}_\mathbf{c}'(\J(\mathscr{F}))$ is
  l.s.c. Hence, by assumption, there exists a continuous mapping
  $\varphi:X\to \mathscr{C}_\mathbf{c}(\J(\mathscr{F}))$ with
  $\theta\leqslant \varphi\leqslant \Phi$. For convenience, set
  $V_F=\mathbf{O}_{1/2}^d(\langle1,F\rangle)$, $F\in
  \mathscr{F}$. Thus, we get an open and discrete in $\J(\mathscr{F})$
  family $\{V_F:F\in \mathscr{F}\}$ with $\langle1,F\rangle\in V_F$,
  $F\in \mathscr{F}$. Moreover, by the definition of $\theta$, we also
  have that $\langle1,F\rangle\in \varphi(x)$, whenever
  $x\in F\in \mathscr{F}$. Since $\varphi$ is compact-valued and both
  l.s.c.\ and u.s.c., this implies that the sets
  $U_F=\varphi^{-1}(V_F)$, $F\in \mathscr{F}$, form a locally finite
  open family in $X$ such that $F\subseteq U_F$, $F\in
  \mathscr{F}$. Accordingly, $X$ is also $\tau$-expandable.
\end{proof}

The proof of Theorem \ref{theorem-choban-sel-v6:3} is almost identical
to that of Theorem \ref{theorem-choban-sel-v6:2}, so we will briefly
point out only the necessary changes. To this end, let us recall that the
order $\ord({\mathscr{W}})$ of a cover $\mathscr{W}$ of $X$ is the
smallest cardinal number $\varkappa$ such that
$|\{W\in \mathscr{W} : x\in W\}| \leq \varkappa$, for every $x \in X$.

\begin{proof}[Proof of Theorem \ref{theorem-choban-sel-v6:3}]
  The implication
  \ref{item:choban-sel-v6:5}$\implies$\ref{item:choban-sel-v6:6} in
  Theorem \ref{theorem-choban-sel-v6:3} is identical to that of
  Theorem \ref{theorem-choban-sel-v6:2}. The only difference is in the
  verification of the locally finite lifting property of the u.s.c.\
  selection $\theta:X\to \mathscr{C}_k(\J(\tau))$, where $k\in
  \N$. This now follows from \cite[Example 3.10]{gutev:00b} because
  $\J(\tau)$ is finite-dimensional. The inverse implication is also
  very similar to that of Theorem \ref{theorem-choban-sel-v6:2}. The
  essential difference is about normality of $X$. To see this, take
  disjoint closed sets $A,B\subset X$ and define an l.s.c.\ mapping
  ${\Phi:X\to \mathscr{C}_\mathbf{c}([-1,1])}$ by $\Phi(x)= [-1,0]$ if
  $x\in A$, $\Phi(x)= [0,1]$ if $x\in B$ and $\Phi(x)= [-1,1]$
  otherwise. Also, define a u.s.c.\ mapping
  $\theta:X\to \mathscr{C}_2([-1,1])$ by $\theta(x)=\{-1,0\}$ if
  $x\in A$, $\theta(x)= \{0,1\}$ if $x\in B$ and $\theta(x)=\{0\}$
  otherwise.  Then $\theta\leqslant \Phi$ and by
  \ref{item:choban-sel-v6:6} of Theorem \ref{theorem-choban-sel-v6:3},
  there exists a continuous mapping
  $\varphi:X\to \mathscr{C}_\mathbf{c}([-1,1])$ with
  $\theta\leqslant \varphi\leqslant \Phi$. Since $\varphi$ is
  connected-valued, this implies that $\varphi(x)=[-1,0]$ for $x\in A$
  and $\varphi(x)=[0,1]$ for $x\in B$. We may now use that $[-1,0]$
  and $[0,1]$ are two different points in the metrizable space
  $\mathscr{C}([-1,1])$. Hence, there are disjoint open sets
  $\Omega_A,\Omega_B\subseteq\mathscr{C}([-1,1])$ such that
  $[-1,0]\in \Omega_A$ and $[0,1]\in \Omega_B$. Since
  $\varphi:X\to \mathscr{C}_\mathbf{c}(Y)\subseteq \mathscr{C}(Y)$ is
  continuous as a usual map in this hyperspace $\mathscr{C}([-1,1])$,
  it follows that
  \[
    U_A=\{x\in X: \varphi(x)\in \Omega_A\}\quad \text{and}\quad
    U_B=\{x\in X: \varphi(x)\in \Omega_B\}
  \]
  are disjoint open subsets of $X$ with $A\subseteq U_A$ and
  $B\subseteq U_B$. Accordingly, $X$ is a normal space. The rest of
  the proof is identical to that of Theorem
  \ref{theorem-choban-sel-v6:2}. Namely, take a locally finite closed
  cover $\mathscr{F}$ of $X$ with $|\mathscr{F}|\leq \tau$ and
  $\ord(\mathscr{F})\leq k$ for some $k\in \N$. Next, as before,
  define a u.s.c.\ mapping $\theta:X\to \mathscr{C}(\J(\mathscr{F}))$
  by
  ${\theta(x)=\{\langle 1,F\rangle: x\in F\ \text{and}\ F\in
    \mathscr{F}\}}$, $x\in X$, and set $\Phi(x)=\J(\mathscr{F})$,
  $x\in X$. Then $\theta:X\to \mathscr{C}_k(\J(\mathscr{F}))$ because
  ${\ord(\mathscr{F})\leq k}$. Hence, by \ref{item:choban-sel-v6:6} of
  Theorem \ref{theorem-choban-sel-v6:3}, there exists a continuous
  mapping $\varphi:X\to \mathscr{C}_\mathbf{c}(\J(\mathscr{F}))$ with
  $\theta\leqslant \varphi\leqslant \Phi$. Just like in the previous
  proof, this implies that the existence of a locally finite open
  cover $\{U_F:F\in\mathscr{F}\}$ of $X$ such that $F\subseteq U_F$,
  $F\in \mathscr{F}$. According to a result of Kat\v{e}tov
  \cite{katetov:58}, $X$ is $\tau$-collectionwise normal.
\end{proof}

Regarding the difference between Theorems
\ref{theorem-choban-sel-v6:1} and \ref{theorem-choban-sel-v6:2}, let
us point out the following special case of these theorems.

\begin{corollary}
  \label{corollary-choban-sel-v10:1}
  For a space $X$, the following are equivalent\textup{:}
  \begin{enumerate}
  \item\label{item:choban-sel-v10:1} $X$ is normal and countably
    paracompact.  
  \item\label{item:choban-sel-v10:2} If
    $\Phi:X\to \mathscr{F}_\mathbf{c}(\J(\omega))$ is l.s.c.,
    $\theta:X\to \mathscr{C}(\J(\omega))$ is u.s.c.\ and
    $\theta\leqslant\Phi$, then there exists a continuous mapping
    ${\varphi:X\to \mathscr{C}_\mathbf{c}(\J(\omega))}$ such that 
    $\theta\leqslant \varphi\leqslant \Phi$.
  \end{enumerate}
\end{corollary}

\begin{proof}
  To see that
  \ref{item:choban-sel-v10:1}$\implies$\ref{item:choban-sel-v10:2},
  let $X$ be normal and countably paracompact,
  $\Phi:X\to \mathscr{F}_\mathbf{c}(\J(\omega))$ be l.s.c.\ and
  $\theta:X\to \mathscr{C}(\J(\omega))$ be a u.s.c.\ selection for
  $\Phi$. By Proposition \ref{proposition-choban-sel-v17:2},
  $\{\Phi(x):x\in X\}$ is uniformly equi-$LC^0$. Hence, by Proposition
  \ref{proposition-Icsvm-vgg-rev:1} and Theorem
  \ref{theorem-choban-sel-v17:1}, there exists a continuous mapping
  ${\varphi:X\to \mathscr{C}_\mathbf{c}(\J(\omega))}$ with
  $\theta\leqslant \varphi\leqslant \Phi$. Since
  \ref{item:choban-sel-v10:2}$\implies$\ref{item:choban-sel-v10:1}
  follows from Theorem \ref{theorem-choban-sel-v6:2}, the proof is
  complete.
\end{proof}

Regarding the role of the family
$\mathscr{C}_\mathbf{c}'(\J(\tau))$ in Theorems
\ref{theorem-choban-sel-v6:2} and \ref{theorem-choban-sel-v6:3}, let
us point out the following characterisation of $\tau$-PF-normal spaces.

\begin{theorem}
  \label{theorem-choban-sel-v14:1}
  For an infinite cardinal number $\tau$ and a space $X$, the
  following conditions are equivalent\textup{:}
  \begin{enumerate}
  \item\label{item:choban-sel-v14:1} $X$ is 
    $\tau$-PF-normal.
  \item\label{item:choban-sel-v14:2} If
    $\Phi:X\to \mathscr{C}_\mathbf{c}(\J(\tau))$ is l.s.c.,
    $\theta:X\to \mathscr{C}(\J(\tau))$ is u.s.c.\ and
    $\theta\leqslant \Phi$, then there exists a continuous mapping
    ${\varphi:X\to \mathscr{C}_\mathbf{c}(\J(\tau))}$ such that
    $\theta\leqslant \varphi\leqslant\Phi$.
  \end{enumerate}
\end{theorem}

\begin{proof}
  Suppose that $X$ is $\tau$-PF-normal and
  $\Phi,\theta:X\to \mathscr{C}(\J(\tau))$ are as in
  \ref{item:choban-sel-v14:2}. Since $\{\Phi(x):x\in X\}$ is uniformly
  equi-$LC^0$, by Proposition \ref{proposition-Icsvm-vgg-rev:2}, the
  associated mapping $\mathscr{C}_\mathbf{c}[\theta,\Phi]$ is
  $H(d)$-s.l.s.f.\ and we may now apply Theorem
  \ref{theorem-choban-sel-v17:1} to get the required intermediate
  mapping $\varphi:X\to \mathscr{C}_\mathbf{c}(Y)$.\smallskip
  
  The inverse implication is also almost identical to that of
  \ref{item:choban-sel-v6:4}$\implies$\ref{item:choban-sel-v6:3} in
  Theorem \ref{theorem-choban-sel-v6:2}. Namely, by Lemma
  \ref{lemma-choban-sel-v12:2}, $X$ is normal. Take a locally finite
  closed cover $\mathscr{F}$ of $X$ with $|\mathscr{F}| \leq \tau$,
  and a point-finite open cover $\{O_F:F\in \mathscr{F}\}$ such that
  $F\subseteq O_F$, $F\in \mathscr{F}$. For every $x\in X$, let
  $\mathscr{F}_x=\{F\in \mathscr{F}: x\in O_F\}$ which is a finite set
  because $\{O_F:F\in \mathscr{F}\}$ is point-finite. Hence, as in the
  proof of Theorem \ref{theorem-choban-sel-v6:1}, we may define an
  l.s.c.\ mapping $\Phi:X\to \mathscr{C}_\mathbf{c}(\J(\mathscr{F}))$
  by $\Phi(x)=\J(\mathscr{F}_x)$, $x\in X$. Next, as in the proof of
  Theorem \ref{theorem-choban-sel-v6:2}, define a u.s.c.\ selection
  $\theta:X\to \mathscr{C}(\J(\mathscr{F}))$ for $\Phi$ by
  $\theta(x)=\{\langle 1,F\rangle: x\in F\ \text{and}\ F\in
  \mathscr{F}\}$, $x\in X$. Thus, by \ref{item:choban-sel-v14:2},
  there exists a continuous mapping
  ${\varphi:X\to \mathscr{C}_\mathbf{c}(\J(\mathscr{F}))}$ with
  $\theta\leqslant \varphi\leqslant \Phi$. Finally, just like before,
  there exists a locally finite open cover $\{U_F:F\in \mathscr{F}\}$
  of $X$ such that $F\subseteq U_F\subseteq O_F$, $F\in
  \mathscr{F}$. According to \cite[Theorem 3.1]{MR1961298}, $X$ is
  $\tau$-PF-normal.
\end{proof}

In the setting of arbitrary $\mathscr{C}'(Y)$-valued l.s.c.\ mappings,
Theorems \ref{theorem-choban-sel-v6:2} and
\ref{theorem-choban-sel-v6:3} are somewhat known. Namely, it was shown
in \cite[Theorem 1.3]{miyazaki:01a} that a space $X$ is normal and
$\tau$-expandable if and only if for every completely metrizable space
$Y$ with $w(Y)\leq \tau$, every l.s.c.\ mapping
$\Phi:X\to \mathscr{C}'(Y)$ and every u.s.c.\ selection
$\theta:X\to \mathscr{C}(Y)$ for $\Phi$, there are mapping
$\varphi,\psi:X\to \mathscr{C}(Y)$ such that $\varphi$ is l.s.c.,
$\psi$ is u.s.c.\ and
$\theta\leqslant \varphi\leqslant \psi\leqslant \Phi$. Such a pair
$(\varphi,\psi)$ of mappings is often called a \emph{Michael
  pair}. So, the pair $(\theta,\Phi)$ admits an intermediate Michael
pair $(\varphi,\psi)$. Similarly, it was shown in \cite[Theorem
1.4]{miyazaki:01a} that a space $X$ is $\tau$-collectionwise normal if
and only if for every finite-dimensional completely metrizable space
$Y$ with $w(Y)\leq \tau$, every l.s.c.\ mapping
$\Phi:X\to \mathscr{C}'(Y)$ and its u.s.c.\ selection
$\theta:X\to \mathscr{C}_k(Y)$, where $k\in\N$, there exists an
intermediate Michael pair $(\varphi,\psi): X\to \mathscr{C}(Y)$ for
$(\theta,\Phi)$. The results in \cite{miyazaki:01a} were not stated in
terms of a cardinal number $\tau\geq\omega$, but the above
characterisations follow easily from the corresponding
proofs. Moreover, the characterisation in \cite[Theorem
1.4]{miyazaki:01a} was stated in the realm of normal spaces, but in
the proof of \cite[Theorem 1.3]{miyazaki:01a} was shown that the
intermediate selection property implies normality. These results can
be extended to $\tau$-PF-normal spaces as well. In fact, all these
characterisations can be obtained using the framework of this paper
and \cite[Theorem 5.1]{gutev:00a}.\medskip

Finally, let us briefly point out that Theorems
\ref{theorem-choban-sel-v6:2}, \ref{theorem-choban-sel-v6:3} and
\ref{theorem-choban-sel-v14:1} also have natural zero-dimensional
versions, see Proposition \ref{proposition-choban-sel-v8:2}. Namely,
we have the following consequences of known results. 

\begin{corollary}
  \label{corollary-choban-sel-v14:1}
  A space $X$ is $\tau$-expandable, normal and has $\dim(X)=0$ if and
  only if for every completely metrizable space $Y$ with
  $w(Y)\leq\tau$, every l.s.c.\ mapping $\Phi:X\to \mathscr{C}'(Y)$
  and every u.s.c.\ selection $\theta:X\to \mathscr{C}(Y)$ for $\Phi$,
  there exists a continuous mapping ${\varphi:X\to \mathscr{C}(Y)}$
  with $\theta\leqslant \varphi\leqslant \Phi$.
\end{corollary}

\begin{proof}
  If $X$ is a $\tau$-expandable normal space with $\dim(X)=0$ and $Y$
  is a metrizable space with $w(Y)\leq \tau$, then every u.s.c.\
  mapping $\theta:X\to \mathscr{C}(Y)$ has the locally lifting
  property \cite[Proposition 2.2]{zbMATH06047968}. Hence, the
  insertion property follows from \cite[Theorem 3.1]{zbMATH06047968}
  and Theorems \ref{theorem-choban-sel-v17:2} and
  \ref{theorem-Icsvm-vgg-rev:2}. Conversely, as in the proof of
  Proposition \ref{proposition-choban-sel-v8:2}, we may show that $X$
  is a normal space with $\dim(X)=0$. Namely, take disjoint closed
  sets $F_1,F_2\subseteq X$ and, for convenience, set $Y=\{0,1,2\}$.
  Next, define an l.s.c.\ mapping $\Phi:X\to 2^Y$ by $\Phi(x)=\{0,i\}$
  if $x\in F_i$, and $\Phi(x)=Y$ otherwise. In contrast to the proof
  of Proposition \ref{proposition-choban-sel-v8:2}, define a u.s.c.\
  selection $\theta:X\to 2^Y$ for $\Phi$ by $\theta(x)=\{0,i\}$ if
  $x\in F_i$, and $\theta(x)=\{0\}$ otherwise. Then by condition,
  there exists a continuous mapping $\varphi:X\to 2^Y$ such that
  $\theta\leqslant\varphi\leqslant\Phi$. In particular,
  $\varphi(x)\neq \{0\}$, $x\in F_1\cup F_2$, and we may define the
  disjoint sets $U_i=\varphi^{-1}(\{i\})\cap \varphi^{\#}(\{0,i\})$,
  $i=1,2$. Since $Y$ is discrete and $\varphi$ is continuous, $U_1$
  and $U_2$ are disjoint clopen sets with $F_i\subseteq U_i$, $i=1,2$,
  so $X$ is a normal space with $\dim(X)=0$. Finally, for each
  completely metrizable space $Y$ with $w(Y)\leq \tau$, by taking
  $\Phi(x)=Y$, $x\in X$, the selection property implies that each
  u.s.c.\ mapping $\theta:X\to \mathscr{C}(Y)$ has the locally finite
  lifting property being a selection for some continuous mapping
  $\varphi:X\to \mathscr{C}(Y)$. Applying \cite[Proposition
  2.2]{zbMATH06047968} once again, this implies that $X$ is also
  $\tau$-expandable.
\end{proof}

\begin{corollary}
  \label{corollary-choban-sel-v14:2}
  A space $X$ is $\tau$-collectionwise normal and has ${\dim(X)=0}$ if
  and only if for every finite-dimensional completely metrizable space
  $Y$ with ${w(Y)\leq\tau}$, every l.s.c.\ $\Phi:X\to \mathscr{C}'(Y)$
  and every u.s.c.\ selection $\theta:X\to \mathscr{C}_k(Y)$ for
  $\Phi$, where $k\in\N$, there exists a continuous mapping
  ${\varphi:X\to \mathscr{C}(Y)}$ with
  $\theta\leqslant \varphi\leqslant \Phi$.
\end{corollary}

\begin{proof}
  Apply the same proof as in Corollary
  \ref{corollary-choban-sel-v14:1}, but now use \cite[Proposition
  2.3]{zbMATH06047968} instead of \cite[Proposition
  2.2]{zbMATH06047968}. 
\end{proof}

Similarly, using \cite[Proposition 2.5]{zbMATH06047968} instead of
\cite[Propositions 2.2 and 2.3]{zbMATH06047968}, we also get the
following result in the setting of $\tau$-PF-normal spaces.

\begin{corollary}
  \label{corollary-choban-sel-v14:3}
  A space $X$ is $\tau$-PF-normal and has $\dim(X)=0$ if and
  only if for every completely metrizable space $Y$ with
  $w(Y)\leq\tau$, every l.s.c.\ mapping $\Phi:X\to \mathscr{C}(Y)$
  and every u.s.c.\ selection $\theta:X\to \mathscr{C}(Y)$ for $\Phi$,
  there exists a continuous mapping ${\varphi:X\to \mathscr{C}(Y)}$
  with $\theta\leqslant \varphi\leqslant \Phi$.
\end{corollary}


\end{document}